\documentclass[11pt,a4paper]{article}

\usepackage[T1]{fontenc}
\usepackage[utf8]{inputenc}
\usepackage[english]{babel}
\usepackage{lmodern}
\usepackage{amsmath,amsthm,amssymb,mathtools,mathrsfs}
\usepackage{a4wide}
\usepackage{enumitem}
\usepackage{float}
\usepackage{tocloft}
\usepackage{xcolor}
\usepackage{url}
\usepackage{hyperref}
\usepackage[hyperpageref]{backref}
\usepackage{orcidlink}
\usepackage{tikz}
\usetikzlibrary{arrows.meta}

\mathtoolsset{showonlyrefs}
\allowdisplaybreaks
\numberwithin{equation}{section}

\definecolor{ForestGreen}{rgb}{0.1,0.6,0.05}
\definecolor{EgyptBlue}{rgb}{0.063,0.1,0.6}
\definecolor{RipeOlive}{HTML}{556B2F}
\hypersetup{
colorlinks=true, linkcolor=EgyptBlue, citecolor=ForestGreen, urlcolor=RipeOlive
}

\newtheorem{theorem}{Theorem}[section]
\newtheorem{proposition}[theorem]{Proposition}
\newtheorem{lemma}[theorem]{Lemma}
\newtheorem{corollary}[theorem]{Corollary}
\theoremstyle{definition}
\newtheorem{assumption}[theorem]{Assumption}
\newtheorem{remark}[theorem]{Remark}

\newcommand{\R}{\mathbb R}
\newcommand{\Om}{\Omega}
\newcommand{\eps}{\varepsilon}
\newcommand{\ph}{\varphi_1}
\newcommand{\Fcal}{\mathcal F}
\newcommand{\Acal}{\mathcal A}
\newcommand{\W}{W_0^{1,p}(\Omega)}
\newcommand{\dual}[2]{\langle #1,#2\rangle}
\newcommand{\weakto}{\rightharpoonup}
\newcommand{\doi}[1]{\href{https://doi.org/#1}{doi:#1}}

\title{\vspace*{-5ex}
On the generalized Fredholm alternative for the $p$-Laplacian\\ with resonant subhomogeneous terms}
\author{Vladimir Bobkov}
\date{}

\AtEndDocument{%
 \par\medskip\bigskip
 \begin{tabular}{@{}l@{}}
 (V.~Bobkov)\\[0.2em]
 \textsc{Institute of Mathematics, Ufa Federal Research Centre, RAS}\\
 \textsc{Chernyshevsky str. 112, 450008 Ufa, Russia}\\[0.3em]
 \orcidlink{0000-0002-4425-0218} 0000-0002-4425-0218\\[0.3em]
 \textit{E-mail addresses}: \texttt{bobkov@matem.anrb.ru},
 \texttt{bobkovve@gmail.com}
 \end{tabular}}

\begin{document}

\maketitle

\vspace*{-5ex}
\begin{abstract}
Let $1\le q<p$ and let $\lambda_1$ be the first eigenvalue of the $p$-Laplacian in a bounded domain $\Omega$.
We study the energy functional
$$
 E_\lambda(u)=\frac1p\left(\int_\Om|\nabla u|^p\,dx
 -\lambda\int_\Om|u|^p\,dx\right)-\Fcal(u),
 \quad u\in W_0^{1,p}(\Om),
$$
where $\Fcal$ is positively $q$-homogeneous and vanishes along the first eigenspace.
Assuming a suitable relation between $\Fcal$ and the $\kappa$-th power of the principal part of $E_{\lambda_1}$ near this eigenspace, we describe the behavior of $E_{\lambda_1}$ according to the relations $p\kappa<q$, $p\kappa=q$, or $p\kappa>q$.
In particular, the functional is unbounded from below in the first case, and has a negative infimum in the last case.
We then study how sufficiently small $q$-homogeneous perturbations of $\Fcal$ influence the geometry of $E_\lambda$.
In this way, we describe assumptions guaranteeing the existence of three critical points in a left neighborhood of $\lambda_1$ and two critical points in a right neighborhood of $\lambda_1$, which indicates an $S$-shaped structure of the solution set.
The results are applied to double-phase functionals and the nonlinear Fredholm alternative.


\end{abstract}

\begin{quote}
\tableofcontents
\addtocontents{toc}{\vspace*{-2ex}}
\end{quote}

\section{Introduction}

Let $\Om\subset\R^N$ be a bounded domain, $N\ge1$, and let $1\le q<p<+\infty$.
For $\lambda \in \R$, consider the energy functional
\begin{equation}\label{eq:intro-energy}
 E_\lambda(u)
 =\frac1p H_\lambda(u)
 -\Fcal(u),
 \quad u\in W_0^{1,p}(\Om),
\end{equation}
where
$$
H_\lambda(u)=\|\nabla u\|_p^p-\lambda\|u\|_p^p, 
$$
and the functional $\Fcal$ is positively $q$-homogeneous. 
When $\Fcal \in C^1(W_0^{1,p}(\Om),\R)$ and it has the integral form $\Fcal(u) = \int_\Om F(x,u,\nabla u)\,dx$, every critical point of $E_\lambda$ satisfies
\begin{align}\label{eq:EL-intro}
 &\int_\Om|\nabla u|^{p-2}\nabla u\cdot\nabla v\,dx
 -\lambda\int_\Om|u|^{p-2}uv\,dx\notag\\
 &\quad=\int_\Om F'_s(x,u,\nabla u)v\,dx
 +\int_\Om \nabla_\xi F(x,u,\nabla u)\cdot\nabla v\,dx,
 \quad v\in W_0^{1,p}(\Om),
\end{align}
where $F'_s$ and $\nabla_\xi F$ correspond to the derivatives of $F=F(x,s,\xi)$ with respect to the second and third variable, respectively. 
Thus, under these assumptions, the critical points are weak solutions of
\begin{equation}\label{eq:intro-problem}
 -\Delta_pu-\lambda|u|^{p-2}u
 =F'_s(x,u,\nabla u)-\operatorname{div}\nabla_\xi F(x,u,\nabla u)
 \quad\text{in }\Om.
\end{equation}
The left-hand side of \eqref{eq:intro-problem} corresponds to the eigenvalue problem for the $p$-Laplacian, and the right-hand side can be understood as its subhomogeneous perturbation.
Therefore, questions of existence and multiplicity of solutions of \eqref{eq:intro-problem} fall into the scope of the generalized Fredholm alternative; see, e.g., \cite{AmbrosettiArcoya1995,AnaneGossez1990,DrabekRobinson1999,GirgTakac2008}.
As in the classical Fredholm alternative (which assumes $p=2$ and $F(x,s,\xi)=f(x)s$), the answers depend strongly on the position of $\lambda$ with respect to the eigenvalues of the $p$-Laplacian and on a relation between the subhomogeneous term and the corresponding eigenfunctions.

In the present work, we are interested in $\lambda$ located in a neighborhood of the first eigenvalue $\lambda_1$.
This eigenvalue can be characterized as
\begin{equation}\label{eq:lambda1}
 \lambda_1=\inf\left\{
 \frac{\int_\Om|\nabla u|^p\,dx}{\int_\Om|u|^p\,dx}:~
 u\in W_0^{1,p}(\Om)\setminus\{0\}\right\}.
\end{equation}
It is known that $\lambda_1>0$ and it is isolated in the set of all eigenvalues of the $p$-Laplacian \cite{anane1987}.
A first eigenfunction $\ph$ is unique up to multiplication by a nonzero constant \cite{Lind}.
Moreover, $\ph\in C^{1,\beta}_{\mathrm{loc}}(\Om)\cap L^\infty(\Omega)$ for some $\beta\in(0,1)$ \cite{Tolksdorf1984}, and it can be chosen so that
\begin{equation}\label{eq:phi-normalization}
 \|\nabla\ph\|_p=1,
 \quad
 \|\ph\|_p^p=\lambda_1^{-1},
\quad
\ph>0\ \text{in }\Om.
\end{equation}
In particular, we have $H_{\lambda_1}(u)\ge0$ for any $u \in \W$, and equality holds if and only if $u \in \R\ph$. 
In other words, $H_{\lambda_1}(u) \geq 0$ is the Poincar\'e inequality with the sharp constant.

In the particular case of the $1$-homogeneous functional $\Fcal(u) = \int_\Om f(x) u\,dx$, 
the equation~\eqref{eq:intro-problem} reads as
\begin{equation}\label{eq:intro-fredholm}
 -\Delta_pu-\lambda|u|^{p-2}u=f(x)
 \quad\text{in }\Om.
\end{equation}
The investigation of \eqref{eq:intro-fredholm} is usually referred to as the nonlinear Fredholm alternative for the $p$-Laplacian, see, e.g.,  \cite{DrabekGirgTakacUlm2004,DrabekHolubova2001,Takac2002,Takac2006}.
In particular, imposing the $L^2(\Omega)$-orthogonality assumption
\begin{equation}\label{eq:intro-orthogonality}
\Fcal(\ph) = \int_\Om f(x) \ph\,dx=0,
\end{equation}
Tak\'a\v c \cite{Takac2002} established the existence of solutions of \eqref{eq:intro-fredholm} at $\lambda=\lambda_1$ for a certain class of source functions $f$ and gave a qualitative description of the corresponding solution set.
The work \cite{Takac2006} further provided at least \textit{three} distinct solutions in a punctured neighborhood of $\lambda_1$ for sufficiently small ``nonorthogonal'' perturbations of $f$. 

Observations somewhat related to \cite{Takac2002,Takac2006} were made by Tanaka and the author in \cite{BobkovTanaka2018,BobkovTanaka2022Multiplicity,BobkovTanaka2023Indefinite} 
for the following two families of subhomogeneous functionals with $q>1$:
$$
\Fcal(u) = \frac{1}{q}\int_\Omega (\mu |u|^q - |\nabla u|^q) \,dx
\quad \text{and} \quad
\Fcal(u) = \frac{1}{q}\int_\Omega a(x) |u|^q \,dx,
$$
where $\mu \in \R$ and $a$ is an indefinite weight.
The critical points of the corresponding energy functionals $E_\lambda$ are (weak) solutions of the equations
\begin{align}
\label{eq:intro-models:1}
 -\Delta_pu-\Delta_qu
 &=\lambda|u|^{p-2}u+\mu|u|^{q-2}u \quad\text{in }\Om,\\
 \label{eq:intro-models:2}
 -\Delta_pu
 &=\lambda|u|^{p-2}u+a(x)|u|^{q-2}u
 \quad \text{in }\Om,
\end{align}
respectively. 
The $(p,q)$-Laplace equation \eqref{eq:intro-models:1} has two spectral parameters $\lambda$ and $\mu$, and the critical point
$$
 (\lambda,\mu)=(\lambda_1,\mu_*), 
 \quad\text{where}\quad
 \mu_*=\frac{\|\nabla\ph\|_q^q}{\|\ph\|_q^q},
$$
is resonant for both parts of $E_\lambda$, that is, $H_{\lambda_1}(\ph) = 0$ and 
\begin{equation}\label{eq:intro-orthogonality:2}
\Fcal(\ph) 
=
 \frac{1}{q}\int_\Omega (\mu_* \ph^q - |\nabla \ph|^q) \,dx
 =
0.
\end{equation}
At this critical spectral point, the geometry of $E_{\lambda_1}$ changes according to $p<2q$, $p=2q$, $p>2q$: the energy is unbounded from below in the first case, while it attains a global infimum in the last case, see \cite{BobkovTanaka2018}.
The subsequent work \cite{BobkovTanaka2022Multiplicity} locates regions in the $(\lambda,\mu)$-plane with two and three positive solutions of \eqref{eq:intro-models:1} near $(\lambda_1,\mu_*)$.
The equation with indefinite subhomogeneous nonlinearity \eqref{eq:intro-models:2} admits a similar change in the structure of the solution set at $\lambda=\lambda_1$ in the cases $p<2q$, $p=2q$, $p>2q$, under the resonance condition
\begin{equation}\label{eq:intro-orthogonality:3}
\Fcal(\ph)=\int_\Om a(x)\ph^q\,dx=0,
\end{equation}
see \cite{BobkovTanaka2023Indefinite}. The existence of multiple nonnegative solutions of \eqref{eq:intro-models:2} in a neighborhood of $\lambda_1$ is also obtained in \cite{BobkovTanaka2023Indefinite} under small nonresonant perturbations of the weight $a$. 

The works \cite{Takac2002,Takac2006} and  \cite{BobkovTanaka2018,BobkovTanaka2022Multiplicity,BobkovTanaka2023Indefinite} 
are the main motivation for the present study: although the considered functionals $\Fcal$ have different forms, their variational geometry and corresponding multiplicity phenomena are significantly similar. 
Our main aim is to describe a set of assumptions on the general functional $\Fcal$ guaranteeing such properties. 
The principal assumption, apart from the positive $q$-homogeneity, will be the resonance condition as in \eqref{eq:intro-orthogonality}, \eqref{eq:intro-orthogonality:2}, \eqref{eq:intro-orthogonality:3}, see Assumption~\ref{ass:resonance} in Section~\ref{sec:framework} below:
\begin{equation}\label{eq:intro-resonance}
 \Fcal(\ph)=\Fcal(-\ph)=0.
\end{equation}
Since we always have $H_{\lambda_1}(\pm \ph)=0$, \eqref{eq:intro-resonance} can be seen as the most ``degenerate'' condition for the geometry of $E_{\lambda_1}$. 
Additional Assumptions~\ref{ass:upper} and \ref{ass:approach} on $\Fcal$ quantify the rate at which $\Fcal$ vanishes relative to the $\kappa$-th power of $H_{\lambda_1}$ near the first eigenspace $\R \ph$.
This yields the change of the geometry of $E_{\lambda_1}$ in the regimes $p\kappa<q$, $p\kappa=q$, $p\kappa>q$, see Theorem~\ref{thm:geometry} below.
The particular case $\kappa=1/2$ corresponds to the mentioned results of \cite{Takac2002,Takac2006} on the equation \eqref{eq:intro-fredholm} and of \cite{BobkovTanaka2018,BobkovTanaka2022Multiplicity,BobkovTanaka2023Indefinite} on \eqref{eq:intro-models:1}, \eqref{eq:intro-models:2}.

The second step in our study is to investigate how the addition of the small positively $q$-homogeneous perturbation $\mu\widetilde{\Fcal}$ influences the geometry of $E_\lambda$ in a neighborhood of $\lambda_1$. 
In the case $p\kappa>q$, we describe regions in the $(\lambda,\mu)$-plane where, under some further assumptions on $\Fcal$ and $\widetilde{\Fcal}$, the perturbed functional $E_{\lambda,\mu}$ has at least two or three critical points, see Theorem~\ref{thm:three} below. 
Vaguely speaking, this indicates that these critical points might form an $S$-shaped bifurcation curve with respect to $\lambda$, see Figure~\ref{fig:S-shaped} below. 

The rest of the present work is organized as follows.
In Section~\ref{sec:framework}, we introduce main assumptions on $\Fcal$.
Section~\ref{sec:geometry} is devoted to the first main result, Theorem~\ref{thm:geometry}, on the geometry of $E_{\lambda}$. 
In Section~\ref{sec:perturbations}, we establish the second main result, Theorem~\ref{thm:three}, on the multiplicity of critical points of the perturbed functional $E_{\lambda,\mu}$. 
Section~\ref{sec:verification} discusses some sufficient conditions for the abstract assumptions from Section~\ref{sec:framework}. 
Finally, we provide some applications of the main results in Section~\ref{sec:applications}. 

\section{Assumptions on \texorpdfstring{$\Fcal$}{F}}\label{sec:framework}

In what follows, we use the notation
\begin{equation}\label{eq:H-and-S}
 S=\{u\in\W:~\|\nabla u\|_p=1\}.
\end{equation}
We state three main assumptions on $\Fcal$ that will determine the geometry of $E_{\lambda}$.
\begin{assumption}\label{ass:resonance}
The functional $\Fcal$ is continuous in $\W$ and positively $q$-homogeneous, the latter meaning that 
\begin{equation}\label{eq:F-homogeneous}
 \Fcal(tu)=t^q\Fcal(u)
 \quad \text{for any}~ 
 t\ge0,~ u\in\W.
\end{equation}
Moreover, $\Fcal$ vanishes along the first eigenspace.
In view of \eqref{eq:F-homogeneous}, this is equivalent to
\begin{equation}\label{eq:resonance}
 \Fcal(\ph)=\Fcal(-\ph)=0.
\end{equation}
\end{assumption}

\begin{assumption}\label{ass:upper}
There exist $\kappa>0$ and $C,\delta>0$ such that
\begin{equation}\label{eq:upper-local}
 \Fcal(w)\le C H_{\lambda_1}(w)^\kappa
\end{equation}
for every $w\in S$ satisfying
$$
 \min_{\sigma\in\{-1,1\}}
 \|\nabla(w-\sigma\ph)\|_p<\delta.
$$
\end{assumption}

\begin{assumption}\label{ass:approach}
Let $\kappa>0$ be as in Assumption~\ref{ass:upper}.
There exist $C,\eps_0>0$ such that, for every $\eps\in(0,\eps_0)$, one can find $w_\eps\in S$ satisfying
\begin{equation}\label{eq:approach-family}
 0<H_{\lambda_1}(w_\eps)\le C\eps
 \quad\text{and}\quad
 \Fcal(w_\eps)\ge C^{-1}\eps^\kappa.
\end{equation}
\end{assumption}

\begin{remark}\label{rem:growth-euler}
The continuity of $\Fcal$ at the origin and \eqref{eq:F-homogeneous} yield the existence of $C>0$ such that
\begin{equation}\label{eq:F-growth-abstract}
 |\Fcal(u)|\le C\|\nabla u\|_p^q,
 \quad u\in\W.
\end{equation}
In particular, $\Fcal$ is bounded on $S$ by $C$. 

If the directional derivative $\dual{\Fcal'(u)}{u}$ exists for some $u \in \W$, then 
the differentiation of \eqref{eq:F-homogeneous} with respect to $t$ at $t=1$ gives the homogeneous Euler identity
\begin{equation}\label{eq:Euler-identity}
 \dual{\Fcal'(u)}{u}=q\Fcal(u).
\end{equation}
\end{remark}

\begin{remark}\label{rem:global-upper}
Assumption~\ref{ass:upper} and the inequality \eqref{eq:F-growth-abstract} imply the existence of $C>0$ such that
\begin{equation}\label{eq:upper-sphere}
 \Fcal(w)\le C H_{\lambda_1}(w)^\kappa
 \quad \text{for any}~ w\in S.
\end{equation}
Indeed, it is not hard to see that, for every neighborhood $U$ of $\{\ph,-\ph\}$ in $S$, we have 
\begin{equation}\label{eq:gap-away}
 \inf_{w\in S\setminus U}H_{\lambda_1}(w)>0.
\end{equation}
Consequently, \eqref{eq:upper-sphere} follows from Assumption~\ref{ass:upper} in a sufficiently small $U$, and from \eqref{eq:gap-away} and \eqref{eq:F-growth-abstract} in $S\setminus U$.
By the homogeneity, we also have
\begin{equation}\label{eq:upper-space}
 \Fcal(u)\le C H_{\lambda_1}(u)^\kappa
 \|\nabla u\|_p^{q-p\kappa},
 \quad u\in\W\setminus\{0\}.
\end{equation}
\end{remark}

\section{Geometry of \texorpdfstring{$E_{\lambda}$}{E lambda}}\label{sec:geometry}
In this section, we establish the first main result of the work -- Theorem~\ref{thm:geometry}. 
We start with an auxiliary statement on the fibering structure of $E_\lambda$.
\begin{lemma}\label{lem:fibering}
Let $u\in\W$ be such that $H_\lambda(u)>0$ and $\Fcal(u)>0$.
Then
\begin{equation}\label{eq:fiber-minimizer}
 t(u)=\left(\frac{q\Fcal(u)}{H_\lambda(u)}\right)^{1/(p-q)}
\end{equation}
is the unique minimizer of $t\mapsto E_\lambda(tu)$ in $(0,+\infty)$, and
\begin{equation}\label{eq:fiber-value}
 \min_{t>0}E_\lambda(tu)
 =-C_{p,q}
 \frac{\Fcal(u)^{p/(p-q)}}{H_\lambda(u)^{q/(p-q)}} < 0,
 \quad\text{where}\quad
 C_{p,q}=q^{q/(p-q)}\frac{p-q}{p}.
\end{equation}
If $H_\lambda(u) \geq 0$ and $\Fcal(u)\le 0$, then $E_\lambda(tu)\ge0$ for any $t\ge0$.
\end{lemma}
\begin{proof}
The positive $p$- and $q$-homogeneity of $H_\lambda$ and $\Fcal$, respectively, gives
\begin{equation}\label{eq:fiber1}
 E_\lambda(tu)=\frac{t^p}{p}H_\lambda(u)-t^q\Fcal(u), \quad t \geq 0.
\end{equation}
By differentiating $t \mapsto E_\lambda(tu)$ with respect to $t$ and analyzing the result, we obtain \eqref{eq:fiber-minimizer}. Substituting $t(u)$ into $E_\lambda(tu)$, we get  \eqref{eq:fiber-value}. The final assertion is immediate from \eqref{eq:fiber1}.
\end{proof}

Let us state the main result on the geometry of $E_{\lambda}$.

\begin{theorem}\label{thm:geometry}
Let Assumption~\ref{ass:resonance} hold.
Then the following assertions are satisfied:
\begin{enumerate}[label=\textup{(\roman*)}]
\item\label{item:geometry-subcritical} 
Let $\lambda<\lambda_1$. Then $E_\lambda$ is coercive and bounded from below.
If $E_\lambda$ is sequentially weakly lower semicontinuous, then its infimum over $\W$ is attained. 
If there exists $u \in \W$ such that $\Fcal(u) > 0$, then 
this infimum is negative.
\item\label{item:geometry-resonant} 
Let $\lambda=\lambda_1$. 
Then the following assertions are satisfied:
\begin{enumerate}[label=\textup{(\alph*)}]
\item\label{item:geometry-subcritical-order} If $p\kappa<q$ and Assumption~\ref{ass:approach} holds, then
$$
 \inf_{\W}E_{\lambda_1}=-\infty.
$$
\item\label{item:geometry-critical-order} If $p\kappa=q$ and Assumptions~\ref{ass:upper} and \ref{ass:approach} hold, then
$$
 -\infty<\inf_{\W}E_{\lambda_1}<0.
$$
\item\label{item:geometry-supercritical-order} If $p\kappa>q$ and Assumptions~\ref{ass:upper} and \ref{ass:approach} hold, then
\begin{equation}\label{eq:resonant-asymptotics}
 -\infty<m_0:=\inf_{\W}E_{\lambda_1}<0
 \quad\text{and}\quad
 \liminf_{\|\nabla u\|_p\to+\infty}E_{\lambda_1}(u)\ge0.
\end{equation}
If, in addition, $E_{\lambda_1}$ is sequentially weakly lower semicontinuous, then $m_0$ is attained. 
\end{enumerate}
\item\label{item:geometry-supercritical} 
Let $\lambda>\lambda_1$. 
Then $\inf_{\W}E_\lambda=-\infty$.
\end{enumerate}
\end{theorem}
\begin{proof}
\ref{item:geometry-subcritical}  If $\lambda<\lambda_1$, then the Poincar\'e inequality $H_{\lambda_1}(u) \geq 0$ and the upper bound \eqref{eq:F-growth-abstract} give
$$
 E_\lambda(u)\ge \frac1p \left(1-\frac{\max\{\lambda,0\}}{\lambda_1}\right)\|\nabla u\|_p^p
 -C\|\nabla u\|_p^q
 \quad \text{for any}~ u\in W_0^{1,p}(\Omega),
$$
which proves the coercivity and boundedness from below.
If $E_\lambda$ is sequentially weakly lower semicontinuous, then the direct method of the calculus of variations justifies that $\inf_{\W}E_\lambda$ is attained.
If there exists $u \in \W$ such that $\Fcal(u) > 0$, then we also have $H_\lambda(u)> H_{\lambda_1}(u) \geq 0$, and Lemma~\ref{lem:fibering} gives $\min_{t>0}E_\lambda(t u) < 0$.
Consequently, the global infimum of $E_\lambda$ is negative.

\ref{item:geometry-resonant} Let $\lambda=\lambda_1$. 
For $u\notin\R\ph$ with $\Fcal(u)>0$, we recall that $H_{\lambda_1}(u)>0$ and denote
\begin{equation}\label{eq:quotient}
 Q(u)=\frac{\Fcal(u)^p}{H_{\lambda_1}(u)^q},
 \quad \text{so that}~ 
 \min_{t>0} E_{\lambda_1}(t u) = -C_{p,q} Q(u)^{1/(p-q)}.
\end{equation}
For the family of functions $\{w_\eps\}$ given by Assumption~\ref{ass:approach}, we get 
\begin{equation}\label{eq:quotient-lower}
\inf_{\W} E_{\lambda_1} 
\leq 
\min_{t>0} E_{\lambda_1}(tw_\eps) 
\leq 
 -C_{p,q} \left(\frac{C^{-p}}{C^q} \, \eps^{p\kappa-q}\right)^{1/(p-q)} < 0.
\end{equation}
This proves the assertion \ref{item:geometry-subcritical-order} by sending $\eps \to 0$, and it gives a negative upper bound for the infimum in \ref{item:geometry-critical-order} and \ref{item:geometry-supercritical-order}.

If $p\kappa\ge q$, then, for any $w\in S$ with $\Fcal(w)>0$, \eqref{eq:upper-sphere} yields
\begin{equation}\label{eq:quotient-upper}
 Q(w)\le C H_{\lambda_1}(w)^{p\kappa-q} \leq C \|\nabla w\|_p^{p(p\kappa-q)} = C.
\end{equation}
Therefore, \eqref{eq:quotient} and the last assertion of Lemma~\ref{lem:fibering} yield a lower bound for $E_{\lambda_1}$ over $\W$.
This proves \ref{item:geometry-critical-order} and also provides a lower bound for $m_0$ in \ref{item:geometry-supercritical-order}.

Assume now that $p\kappa>q$, and let us justify the second inequality in \eqref{eq:resonant-asymptotics}. 
Let $\{u_n\}$ be any sequence such that $\|\nabla u_n\|_p\to+\infty$. 
Denote $w_n = u_n/\rho_n$, where $\rho_n=\|\nabla u_n\|_p$, so that $w_n\in S$.
If $\liminf_{n \to +\infty}E_{\lambda_1}(u_n) \geq 0$, then there is nothing to prove.
Otherwise, there exists a subsequence (which we denote by the same index) along which $\{E_{\lambda_1}(u_n)\}$ is bounded from above by some constant $C_1 < 0$. 
Then the last assertion of Lemma~\ref{lem:fibering} implies that $\Fcal(w_n)>0$ for any sufficiently large $n$.
Recalling from Remark~\ref{rem:growth-euler} that $\Fcal$ is bounded on $S$ by $C>0$, we get
$$
 \frac{\rho_n^p}{p}H_{\lambda_1}(w_n)
 \leq C_1 + \rho_n^q \Fcal(w_n) 
 \leq C_1 + C \rho_n^q,
$$
which yields $H_{\lambda_1}(w_n)\to 0$. 
Since $p\kappa-q>0$ and recalling that $\Fcal(w_n)>0$, we derive from the first inequality in \eqref{eq:quotient-upper} that $Q(w_n)\to0$.
Thus, by \eqref{eq:quotient}, the minimum of $E_{\lambda_1}(tw_n)$ over $t>0$ tends to zero, which contradicts our choice of $\{u_n\}$. 
This proves the second inequality in \eqref{eq:resonant-asymptotics}.

As a consequence of the latter fact and $m_0<0$, any minimizing sequence for $m_0$ is bounded in $\W$.
Hence, if $E_{\lambda_1}$ is sequentially weakly lower semicontinuous, then the global infimum is attained by a nonzero function.

\ref{item:geometry-supercritical}  
Finally, if $\lambda>\lambda_1$, then, by \eqref{eq:resonance} and the normalization \eqref{eq:phi-normalization} of $\ph$, we get
$$
 E_\lambda(t\ph)=\frac{t^p}{p}
 \left(1-\frac{\lambda}{\lambda_1}\right) \to -\infty
 \quad \text{as}~ t \to +\infty.
$$
The proof is complete.
\end{proof}

\begin{remark}\label{rem:attain}
Clearly, if $\Fcal\in C^1(\W,\R)$, then so is $E_\lambda$. 
Consequently, if the global infimum of $E_\lambda$ is attained, then any minimizer is a critical point of $E_\lambda$.
\end{remark}

\begin{remark}\label{rem:borderline}
In the borderline case $p\kappa=q$, Theorem~\ref{thm:geometry} provides no answer on the attainability of the infimum.
In the classical linear Fredholm alternative, which corresponds to $p=2$, $q=1$, $\kappa=1/2$ (see Corollary~\ref{cor:linear-source} below), the global infimum is attained, but the set of minimizers is unbounded in $W_0^{1,2}(\Om)$. 
In the nonlinear Fredholm alternative \eqref{eq:intro-fredholm}, one has $p \neq 2$,  $q=1$, $\kappa=1/2$ (by the same Corollary~\ref{cor:linear-source}, see also Section~\ref{sec:nonlinear-fredholm}), and hence the borderline case $p\kappa=q$ 
\textit{does not arise}. 
For the problems \eqref{eq:intro-models:1} and \eqref{eq:intro-models:2} considered in \cite{BobkovTanaka2018,BobkovTanaka2023Indefinite}, one has $q>1$ and $\kappa=1/2$ (see Section~\ref{sec:double-phase} below), and the attainability question remains open at the threshold $p=2q$.
\end{remark}

\begin{remark}\label{rem:superhomogeneous-range}
If one considers the superhomogeneous case $q>p$ instead of $q<p$, then the fibering geometry is reversed: along every direction $tu$ with $\Fcal(u)>0$ we have $E_\lambda(tu)\to-\infty$ as $t\to+\infty$, and the map $t \mapsto E_\lambda(tu)$ 
has at most one critical point, which is a point of maximum rather than minimum.
In particular, if the set $\{u:\,\Fcal(u)>0\}$ is nonempty, then 
the infimum of $E_\lambda$ over $\W$ is always $-\infty$, and hence the present investigation cannot be extended directly to $q>p$.
\end{remark}

\section{Critical points under perturbations}\label{sec:perturbations}

In this section, we state and prove the second main result of the work -- Theorem~\ref{thm:three}, which investigates how the geometry of $E_\lambda$ changes under perturbations of $\Fcal$. 
Let $\widetilde{\Fcal}$ be a positively $q$-homogeneous functional, and set
\begin{equation}\label{eq:perturbed-functionals}
 \Fcal_\mu=\Fcal+\mu \widetilde{\Fcal}
 \quad\text{and}\quad
 E_{\lambda,\mu}(u)=\frac1pH_\lambda(u)-\Fcal_\mu(u).
\end{equation}
Assume that either $\widetilde{\Fcal}(\ph)>0$ or $\widetilde{\Fcal}(-\ph)>0$, and fix $\sigma\in\{-1,1\}$ such that
\begin{equation}\label{eq:imbalance}
 b:=\widetilde{\Fcal}(\sigma\ph)>0.
\end{equation}
For $\mu>0$, define
\begin{equation}\label{eq:positive-set}
 \Acal_\mu=\{u\in\W:~\Fcal_\mu(u)>0\}.
\end{equation}

\begin{theorem}\label{thm:three}
Let Assumptions~\ref{ass:resonance}, \ref{ass:upper}, \ref{ass:approach} be satisfied, and $p\kappa>q$.
Let, for some $\mu_0>0$, $\Fcal_\mu \in C^1(\W,\R)$ for every $\mu\in[0,\mu_0]$.
If $q>1$, assume that, for every $\mu\in(0,\mu_0]$, either $\Acal_\mu$ is path-connected, or $\Fcal_\mu$ is even and, for every $u\in\Acal_\mu$, one of $u$ and $-u$ can be joined to $\sigma\ph$ by a continuous path in $\Acal_\mu$.
Assume also that, for some $\delta_0>0$,
\begin{enumerate}[label=\textup{(\alph*)}]
\item\label{thm:three:as1} $E_{\lambda,\mu}$ is sequentially weakly lower semicontinuous whenever $|\lambda-\lambda_1|\le\delta_0$ and $\mu \in [0,\mu_0]$;
\item\label{thm:three:as2} $E_{\lambda,\mu}$ satisfies the Palais--Smale condition whenever $0<|\lambda-\lambda_1|\le\delta_0$ and $\mu \in [0,\mu_0]$.
\end{enumerate}
Then there exist $\mu_* \in (0,\mu_0]$, $\delta_* \in (0,\delta_0]$, and $c_*>0$ such that the following assertions hold (see Figure~\ref{fig:parameter-regions}):
\begin{enumerate}[label=\textup{(\roman*)}]
\item\label{item:three-left} If $\mu \in (0, \mu_*)$ and $\lambda \in (\lambda_1-c_*\mu^{p/q}, \lambda_1)$,
then $E_{\lambda,\mu}$ has at least three distinct nonzero critical points: a local minimizer $u_1$, a global minimizer $u_2$, and a critical point $u_3$ of mountain-pass type. 
Moreover, for a fixed $\mu$, $E_{\lambda,\mu}(u_2) \to -\infty$ as $\lambda \nearrow \lambda_1$. 
\item\label{item:two-right} 
If $\mu \in (0, \mu_*)$ and $\lambda \in (\lambda_1, \lambda_1+\delta_*)$,
then $E_{\lambda,\mu}$ has at least two distinct nonzero critical points: a local minimizer $u_1$ and a critical point $u_4$ of mountain-pass type.
\end{enumerate}
Furthermore, we have $E_{\lambda,\mu}(u_i)<0$ for $i=1,2,3,4$.
\end{theorem}

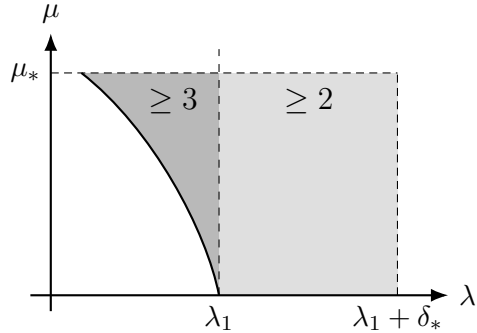
\begin{figure}[H]
\centering
\begin{tikzpicture}[x=1.35cm,y=1.05cm,>=Latex]
 \fill[gray!55]
  (3.6,0.45) .. controls (3.48,1.25) and (2.95,2.55) .. (2.25,3.25)
  -- (3.6,3.25) -- cycle;
 \fill[gray!25] (3.6,0.45) rectangle (5.35,3.25);
 \draw[thick,->] (1.75,0.45)--(5.85,0.45) node[right] {$\lambda$};
 \draw[thick,->] (1.95,0.25)--(1.95,3.75) node[above] {$\mu$};
 \draw[dashed] (3.6,0.45)--(3.6,3.55);
 \draw[thick]
  (3.6,0.45) .. controls (3.48,1.25) and (2.95,2.55) .. (2.25,3.25);
 \draw[densely dashed] (1.95,3.25)--(5.35,3.25);
 \draw[densely dashed] (5.35,0.45)--(5.35,3.25);
 \node[below] at (3.6,0.45) {$\lambda_1$};
 \node[below] at (5.35,0.45) {$\lambda_1+\delta_*$};
 \node[left] at (1.95,3.25) {$\mu_*$};
 \node at (3.15,2.9) {\large $\ge3$};
 \node at (4.48,2.9) {\large $\ge2$};
\end{tikzpicture}
\caption{Schematic plot of the $(\lambda,\mu)$-plane: darker gray -- at least three critical points, lighter gray -- at least two critical points.}
\label{fig:parameter-regions}
\end{figure}

\begin{proof}
By Theorem~\ref{thm:geometry}  \ref{item:geometry-resonant} \ref{item:geometry-supercritical-order} and Remark~\ref{rem:attain}, the unperturbed functional $E_{\lambda_1,0}$ has a nonzero global minimizer $u_0 \in \W$ at the level $m_0<0$, and
$$
 \liminf_{\|\nabla u\|_p\to+\infty}E_{\lambda_1,0}(u)\ge0.
$$
Let us choose a sufficiently large $R>\|\nabla u_0\|_p$ such that
\begin{equation}\label{eq:base-barrier}
 \inf_{\|\nabla u\|_p=R}E_{\lambda_1,0}(u)>\frac{m_0}{2}.
\end{equation}
Observe that Remark~\ref{rem:growth-euler} holds for $\widetilde{\Fcal}$, which gives the existence of $C_1>0$ such that $|\widetilde{\Fcal}(u)| \leq C_1$ for any $u$ with $\|\nabla u\|_p \leq R$.
Therefore, by the Poincar\'e inequality $H_{\lambda_1}(u) \geq 0$, we obtain the uniform estimate
\begin{equation}\label{eq:uniform-perturbation}
 |E_{\lambda,\mu}(u)-E_{\lambda_1,0}(u)|
 \le \frac{|\lambda-\lambda_1|}{p\lambda_1}R^p + C_1\mu
 \quad \text{for any $u$ such that}~ \|\nabla u\|_p \leq R.
\end{equation}
Thus, there exist $\mu_*\in(0,\mu_0]$ and $\delta_*\in(0,\delta_0]$, such that \eqref{eq:base-barrier} and \eqref{eq:uniform-perturbation} imply
\begin{equation}\label{eq:perturbed-barrier}
 \inf_{\|\nabla u\|_p=R}E_{\lambda,\mu}(u)
 >\inf_{\|\nabla u\|_p\le R}E_{\lambda,\mu}(u)
\end{equation}
whenever $\mu \in [0, \mu_*)$ and $\lambda \in (\lambda_1-\delta_*, \lambda_1+\delta_*)$. 
Since $E_{\lambda,\mu}$ is sequentially weakly lower semicontinuous, its infimum over the ball $\{\|\nabla u\|_p\le R\}$ is attained at an interior point $u_1$.
Hence, $u_1$ is a local minimizer and a critical point of $E_{\lambda,\mu}$.
By decreasing $\mu_*$ and $\delta_*$ if necessary, we further guarantee that $E_{\lambda,\mu}(u_1)<0$.
At every critical point $u$ of $E_{\lambda,\mu}$, the homogeneous Euler identity \eqref{eq:Euler-identity}, applied to $\Fcal_\mu$, gives
\begin{equation}\label{eq:perturbed-critical-identities}
 H_\lambda(u)=q\Fcal_\mu(u)
 \quad\text{and}\quad
 E_{\lambda,\mu}(u)=-\frac{p-q}{p}\Fcal_\mu(u).
\end{equation}
Thus, we have $u_1\in\Acal_\mu$.
In the even-alternative assumption on $\Acal_\mu$, both $u_1$ and $-u_1$ are local minimizers, and we fix the sign of $u_1$ so that it can be continuously joined to $\sigma\ph$ in $\Acal_\mu$.

\ref{item:three-left} Let $\mu \in (0,\mu_*)$ and $\lambda \in (\lambda_1-\delta_*,\lambda_1)$. 
Arguing as in the proof of Theorem~\ref{thm:geometry} \ref{item:geometry-subcritical}, 
it is not hard to see that $E_{\lambda,\mu}$ is bounded from below and coercive, which yields the existence of its global minimizer $u_2$. 
By the normalization \eqref{eq:phi-normalization} of $\ph$, \eqref{eq:resonance}, and \eqref{eq:imbalance}, we have 
$$
 H_\lambda(\sigma\ph)=\frac{\lambda_1-\lambda}{\lambda_1}
 \quad\text{and}\quad
 \Fcal_\mu(\sigma\ph)=\mu b.
$$
Therefore, the fibering formula \eqref{eq:fiber-value} implies that
\begin{equation}\label{eq:first-eigenspace-value}
 \min_{t>0}E_{\lambda,\mu}(t\sigma\ph)
 =-C_{p,q}(\mu b)^{p/(p-q)}
 \left(\frac{\lambda_1-\lambda}{\lambda_1}\right)^{-q/(p-q)}.
\end{equation}
Let us compare this value with $E_{\lambda,\mu}(u_1)$. 
Recalling that $E_{\lambda_1,0}(u)\ge m_0$ for every $u\in\W$,  \eqref{eq:uniform-perturbation} gives the uniform lower bound
\begin{equation}\label{eq:ball-lower-bound}
 E_{\lambda,\mu}(u)
 \ge m_0-\frac{\delta_*}{p\lambda_1}R^p-C_1\mu_*
 =:-M
 \quad\text{for any }u \text{ such that }\|\nabla u\|_p\le R,
\end{equation}
where $M>0$ is independent of $\mu \in (0,\mu_*)$ and $\lambda \in (\lambda_1-\delta_*,\lambda_1)$. 
Therefore, we see that 
\eqref{eq:first-eigenspace-value} is less than $-M$ provided
$\mu\in(0,\mu_*)$ and $\lambda\in(\lambda_1-c_*\mu^{p/q},\lambda_1)$, where 
$$
c_* = \min\left\{\frac{\delta_*}{\mu_*^{p/q}}, \lambda_1 b^{p/q} \left(\frac{C_{p,q}}{ M}\right)^{(p-q)/q}\right\}.
$$
Combining this fact with \eqref{eq:ball-lower-bound} and the global minimality of $u_2$, we obtain
\begin{equation}\label{eq:separated-minima}
 E_{\lambda,\mu}(u_2)
 \le\min_{t>0}E_{\lambda,\mu}(t\sigma\ph)
 <-M\le\inf_{\|\nabla u\|_p\le R}E_{\lambda,\mu}(u)
 =E_{\lambda,\mu}(u_1)<0.
\end{equation}
Consequently, we have $\|\nabla u_2\|_p>R$, and the global minimizer $u_2$ is distinct from the local minimizer $u_1$. 
Moreover, for a fixed $\mu$, we see from \eqref{eq:first-eigenspace-value} and \eqref{eq:separated-minima} that $E_{\lambda,\mu}(u_2) \to -\infty$ as $\lambda \nearrow \lambda_1$. 

For later purposes, we note that since $u_2$ is a critical point and $E_{\lambda,\mu}(u_2)<0$, \eqref{eq:perturbed-critical-identities} gives $u_2\in\Acal_\mu$.
Under the even-alternative assumption on $\Acal_\mu$, both $u_2$ and $-u_2$ are global minimizers, and we fix the sign of $u_2$ so that it can be continuously joined to $\sigma\ph$ in $\Acal_\mu$.

Let us now define
$$
 \Gamma_1=\{\gamma\in C([0,1],\W):~\gamma(0)=u_1, \gamma(1)=u_2\}
 \quad\text{and}\quad
 c_1=\inf_{\gamma\in\Gamma_1}\max_{t\in[0,1]}E_{\lambda,\mu}(\gamma(t)).
$$
By construction, every path in $\Gamma_1$ crosses the sphere $\{\|\nabla u\|_p=R\}$.
Thus, \eqref{eq:perturbed-barrier} and \eqref{eq:separated-minima} yield
\begin{equation}\label{eq:left-mountain-level}
 c_1\ge\inf_{\|\nabla u\|_p=R}E_{\lambda,\mu}(u)
 >E_{\lambda,\mu}(u_1)>E_{\lambda,\mu}(u_2)
\end{equation}
in the considered range of $\lambda$ and $\mu$. 
Let us justify that $c_1<0$. 
For this purpose, we construct a path in $\Gamma_1$ along which $E_{\lambda,\mu}$ is negative.
Note that if $q=1$, then the set $\Acal_\mu$ is path-connected.
Indeed, since $\Fcal_\mu$ is positively homogeneous and differentiable by our assumptions, we have $\Fcal_\mu(0)=0$ and 
\begin{equation}\label{eq:deriv1x}
 \Fcal_\mu(u)
 =\lim_{t \to 0^+}\frac{\Fcal_\mu(tu)-\Fcal_\mu(0)}t
 =\dual{\Fcal_\mu'(0)}{u}
 \quad\text{for any }u\in\W.
\end{equation}
Thus, $\Fcal_\mu$ is a continuous linear functional, and it is nonzero because $\Fcal_\mu(\sigma\ph)=\mu b>0$ by \eqref{eq:resonance} and \eqref{eq:imbalance}. 
Consequently, $\Acal_\mu$ is an open half-space and hence is convex and path-connected.

If $\Acal_\mu$ is assumed path-connected, then it contains a continuous path $\eta:[0,1]\to\Acal_\mu$ joining $u_1$ and $u_2$.
Otherwise, under the even-alternative assumption on $\Acal_\mu$, we concatenate the chosen paths from $u_1$ and $u_2$ to $\sigma\ph$ to obtain such a path $\eta$.
The continuity of $\Fcal_\mu$ and $H_\lambda$ on the compact set $\eta([0,1])$ gives
$$
 a_\eta:=\min_{t\in[0,1]}\Fcal_\mu(\eta(t))>0
 \quad\text{and}\quad
 d_\eta:=\max_{t\in[0,1]} H_\lambda(\eta(t))<+\infty.
$$
Now we choose $\tau\in(0,1)$ so small that $\tau^{p-q}d_\eta/p<a_\eta$.
Then the homogeneity implies that 
\begin{equation}\label{eq:scaled-negative-path}
 E_{\lambda,\mu}(\tau\eta(t))
 \le\tau^q\left(\frac{\tau^{p-q}}p d_\eta-a_\eta\right)<0
 \quad\text{for any }t\in[0,1].
\end{equation}
Moreover, if $u$ is a critical point with $E_{\lambda,\mu}(u)<0$, then \eqref{eq:perturbed-critical-identities} gives
\begin{equation}\label{eq:critical-scaling-negative}
 E_{\lambda,\mu}(tu)=\Fcal_\mu(u)
 \left(\frac qp t^p-t^q\right)<0,
 \quad\text{for any }t\in(0,1].
\end{equation}
Thus, concatenating the segment from $u_1$ to $\tau u_1$, the path $\tau\eta$, and the segment from $\tau u_2$ to $u_2$, we obtain a continuous path in $\Gamma_1$ along which $E_{\lambda,\mu}$ is negative.
Consequently, we get $c_1<0$.
The mountain-pass lemma and the Palais--Smale condition provide a critical point $u_3$ at the level $c_1$, see, e.g., \cite{Rabinowitz1986}. 
By \eqref{eq:left-mountain-level}, the critical points $u_1,u_2,u_3$ are distinct and nonzero. 
This proves \ref{item:three-left}.

\ref{item:two-right} 
Let $\mu \in (0,\mu_*)$ and $\lambda \in (\lambda_1, \lambda_1+\delta_*)$.
By \eqref{eq:phi-normalization}, \eqref{eq:resonance}, and \eqref{eq:imbalance}, we have
\begin{equation}\label{eq:right-eigenspace}
 E_{\lambda,\mu}(t\sigma\ph)
 =-\frac{\lambda-\lambda_1}{p\lambda_1}t^p-\mu b t^q<0
 \quad\text{for any }t>0.
\end{equation}
This expression tends to $-\infty$ as $t\to+\infty$.
Let us choose a sufficiently large $T>\max\{R,1\}$ such that
$$
 E_{\lambda,\mu}(T\sigma\ph)<E_{\lambda,\mu}(u_1).
$$
As above, we define
$$
 \Gamma_2=\{\gamma\in C([0,1],\W):~\gamma(0)=u_1,\ \gamma(1)=T\sigma\ph\}
 \quad \text{and} \quad
 c_2=\inf_{\gamma\in\Gamma_2}\max_{t\in[0,1]}E_{\lambda,\mu}(\gamma(t)).
$$
Since $\|\nabla u_1\|_p<R<T=\|\nabla (T\sigma\ph)\|_p$, every path in $\Gamma_2$ crosses the sphere $\{\|\nabla u\|_p=R\}$.
Consequently, \eqref{eq:perturbed-barrier} gives
\begin{equation}\label{eq:right-mountain-level}
 c_2\ge\inf_{\|\nabla u\|_p=R}E_{\lambda,\mu}(u)
 >E_{\lambda,\mu}(u_1).
\end{equation}
It remains to verify that $c_2<0$.
To this end, we argue in much the same way as above. 
The identity \eqref{eq:perturbed-critical-identities} and $E_{\lambda,\mu}(u_1)<0$ show that $u_1\in\Acal_\mu$, and \eqref{eq:right-eigenspace} gives $\sigma\ph\in\Acal_\mu$.
If $\Acal_\mu$ is path-connected, then it contains a continuous path $\eta:[0,1]\to\Acal_\mu$ joining $u_1$ and $\sigma\ph$.
Otherwise, under the even-alternative assumption on $\Acal_\mu$, we recall that the sign of $u_1$ was chosen above so that such a path $\eta$ exists.
As in \eqref{eq:scaled-negative-path}, we find $\tau\in(0,1)$ such that $E_{\lambda,\mu}(\tau\eta(t))<0$ for every $t\in[0,1]$.
Further using \eqref{eq:critical-scaling-negative}, we concatenate the segment from $u_1$ to $\tau u_1$, the path $\tau\eta$, and the segment from $\tau\sigma\ph$ to $T\sigma\ph$ to get a path $\gamma\in\Gamma_2$ along which $E_{\lambda,\mu}(\gamma(t))<0$ for every $t\in[0,1]$.
This gives $c_2<0$.
The mountain-pass lemma and the Palais--Smale condition provide a critical point $u_4$ at the level $c_2$.
By \eqref{eq:right-mountain-level}, $u_4$ and $u_1$ are distinct and nonzero. 
This proves \ref{item:two-right}.
\end{proof}

\begin{remark}\label{rem:path-connectedness}
The assumptions on $\Acal_\mu$ in Theorem~\ref{thm:three} guarantee that the constructed minimax levels are negative.
Without such assumptions, the mountain-pass lemma still gives a critical point, but the argument does not exclude the zero critical point, since $E_{\lambda,\mu}'(0)=0$ when $q>1$.
\end{remark}

\begin{remark}\label{rem:S-shaped}
For each fixed sufficiently small $\mu>0$, Theorem~\ref{thm:three} indicates that the obtained critical points of $E_{\lambda,\mu}$ might form an $S$-shaped bifurcation diagram in the $(\lambda,\|\nabla u\|_p)$-plane, see Figure~\ref{fig:S-shaped}. 
This is supported by the results from \cite{BobkovTanaka2022Multiplicity,BobkovTanaka2023Indefinite}. 
However, we do not investigate the actual continuity of the corresponding branches in the present work. 
We also refer to \cite{CarvalhoIlyasovSantos2022,KormanLi1999,WangYeh2008} for results on the $S$-shaped bifurcation diagrams for particular problems. 
For general three-critical-points theorems, see, e.g., \cite{AvernaBonanno2003,Ricceri2000}.
\end{remark}

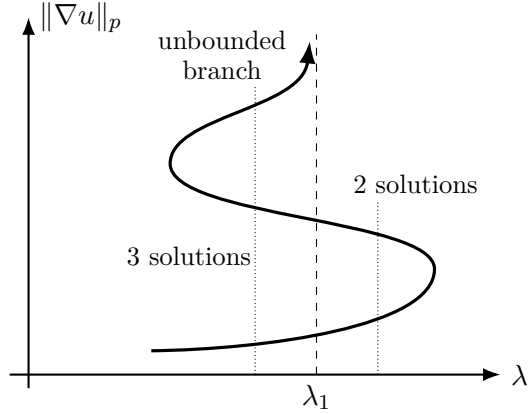
\begin{figure}[!htb]
\centering
\begin{tikzpicture}[x=1.25cm,y=0.9cm,>=Latex]
 \draw[thick,->] (0.55,0.45)--(5.75,0.45) node[right] {$\lambda$};
 \draw[thick,->] (0.75,0.25)--(0.75,5.7) node[right] {$\|\nabla u\|_p$};
 \draw[dashed] (3.8,0.45)--(3.8,5.45);
 \draw[densely dotted] (3.15,0.45)--(3.15,4.7);
 \draw[densely dotted] (4.45,0.45)--(4.45,3.);
 \draw[very thick,->]
  (2.05,0.8)
  .. controls (3.25,0.82) and (5.05,1.15) .. (5.05,2.0)
  .. controls (5.05,2.75) and (2.25,2.75) .. (2.25,3.55)
  .. controls (2.25,4.25) and (3.65,4.35) .. (3.73,5.35);
 \node[below] at (3.8,0.45) {$\lambda_1$};
 \node[align=center] at (2.8,5.15) {\small unbounded\\[-0.2em]\small branch};
 \node at (2.45,2.2) {\small $3$ solutions};
 \node at (4.85,3.25) {\small $2$ solutions};
\end{tikzpicture}
\caption{A schematic $S$-shaped bifurcation diagram in a neighborhood of $\lambda_1$.}
\label{fig:S-shaped}
\end{figure}

\section{Sufficient conditions}\label{sec:verification}

In this section, we give sufficient conditions for the validity of Assumptions~\ref{ass:upper} and~\ref{ass:approach}. 
To this end, we recall the improved Poincar\'e (Friedrichs) inequality of Fleckinger-Pell\'e \& Tak\'a\v{c} \cite{FleckingerTakac2002}, which is used with related purposes in \cite{BobkovTanaka2018,BobkovTanaka2023Indefinite,Takac2002,Takac2006}.
Let us introduce the following additional assumption on the domain. 

\begin{assumption}\label{ass:domain-regularity}
If $p>2$, then $\Om$ is of class $C^{1,\alpha}$ for some $\alpha\in(0,1)$.
\end{assumption}

Every $u\in\W$ has a unique decomposition
\begin{equation}\label{eq:decomposition}
 u=t\ph+v,
 \quad 
 \text{where}~ t=\frac{\int_\Om\ph^{p-1}u\,dx}{\int_\Om\ph^p\,dx}
 \quad \text{and} \quad 
 \int_\Om\ph^{p-1}v\,dx=0.
\end{equation}
For $p>2$, define
\begin{equation}\label{eq:weighted-norm}
 \|v\|_{\ph}^2
 =\int_\Om|\nabla\ph|^{p-2}|\nabla v|^2\,dx, 
\end{equation}
and denote by $D_{\ph}$ the completion of $W_0^{1,p}(\Om)$ with respect to this norm. 
In the case $p=2$, we set $\|v\|_{\ph}=\|\nabla v\|_2$ and identify $D_{\ph}$ with $W_0^{1,2}(\Omega)$.

For $p=2$, the decomposition \eqref{eq:decomposition} can be used to get the following improved Poincar\'e inequality:
\begin{equation}\label{eq:improved-poincare0}
 H_{\lambda_1}(t\ph+v)
 \ge \frac{\lambda_2-\lambda_1}{\lambda_2} \|\nabla v\|_2^2,
\end{equation}
where $\lambda_2$ is the second eigenvalue of the Laplacian in $\Omega$, see, e.g., \cite{BobkovKolonitskii2023,FleckingerTakac2002}.

For $p>2$ and under Assumption~\ref{ass:domain-regularity}, the improved Poincar\'e  inequality \cite{FleckingerTakac2002} in the form of \cite[Theorem~1.1]{BobkovKolonitskii2023} states the existence of a constant $C_P>0$ such that, for the decomposition \eqref{eq:decomposition},
\begin{equation}\label{eq:improved-poincare}
 H_{\lambda_1}(t\ph+v)
 \ge C_P\left(|t|^{p-2}\|v\|_{\ph}^2
 +\|\nabla v\|_p^p\right).
\end{equation}
Under the same assumptions, the continuous embedding of $D_{\ph}$ into $L^2(\Om)$ follows from \cite[Lemma~2.2\,\textup{(i)}]{BobkovKolonitskii2023}, implying 
\begin{equation}\label{eq:weighted-L2}
 \|v\|_2\le C\|v\|_{\ph}
 \quad\text{for any }v\in W_0^{1,p}(\Om).
\end{equation}

\begin{proposition}\label{prop:upper-criterion}
Let $p \geq 2$ and let Assumptions \ref{ass:resonance} and \ref{ass:domain-regularity} hold.
Assume that there exist $\kappa, C,\delta>0$ such that
\begin{equation}\label{eq:upper-A}
 \Fcal(t\sigma\ph+v) 
 \le C\left(\|v\|_{\ph}^2+\|\nabla v\|_p^p\right)^\kappa
\end{equation}
whenever $\sigma\in\{-1,1\}$, $|t-1|<\delta$, and $v \in \W$ satisfies $\int_\Om\ph^{p-1}v\,dx=0$ and $\|\nabla v\|_p<\delta$.
Then Assumption~\ref{ass:upper} holds for $\kappa$. 

For the validity of \eqref{eq:upper-A}, it is sufficient to assume that $\Fcal$ is continuously differentiable in a neighborhood of $\{\ph,-\ph\}$ and that
\begin{equation}\label{eq:derivative-A}
 \dual{\Fcal'(t\sigma\ph+\theta v)}{v}
 \le C\left(\|v\|_{\ph}^2+\|\nabla v\|_p^p\right)^\kappa,
 \quad 0\le\theta\le1,
\end{equation}
uniformly for $t,\sigma,v$ as above.
\end{proposition}

\begin{proof}
Let us decrease $\delta$ so that $\delta\le1/2$.
Let $w\in S$ satisfy $\|\nabla(w-\sigma\ph)\|_p<\delta_1$ for some $\sigma\in\{-1,1\}$ and  $\delta_1 \in (0,\delta/2)$.
In the decomposition $w=t\sigma\ph+v$, \eqref{eq:phi-normalization} and \eqref{eq:decomposition} give
$$
 t=\sigma\lambda_1\int_\Om\ph^{p-1}w\,dx
 \quad\text{and}\quad
 \int_\Om\ph^{p-1}v\,dx=0.
$$
By the H\"older and Poincar\'e inequalities, we have
$$
 |t-1|\le\lambda_1\|\ph\|_p^{p-1}\|w-\sigma\ph\|_p
 \le\|\nabla(w-\sigma\ph)\|_p<\delta_1.
$$
Consequently, $t>1/2$ and $\|\nabla v\|_p\le\|\nabla(w-\sigma\ph)\|_p+|t-1|<2\delta_1<\delta$, so that \eqref{eq:upper-A} applies to $v$.
Moreover, \eqref{eq:improved-poincare} for $p>2$ and \eqref{eq:improved-poincare0} for $p=2$ give
\begin{equation}\label{eq:h1ower}
 H_{\lambda_1}(w)\ge C_0\left(\|v\|_{\ph}^2+\|\nabla v\|_p^p\right),
\end{equation}
where $C_0=C_P2^{2-p}$ if $p>2$, and $C_0=(\lambda_2-\lambda_1)/(2\lambda_2)$ if $p=2$.
Combining \eqref{eq:h1ower} with \eqref{eq:upper-A}, we obtain
$$
 \Fcal(w)\le C C_0^{-\kappa}H_{\lambda_1}(w)^\kappa,
$$
which proves Assumption~\ref{ass:upper}.

Finally, if $\Fcal$ is continuously differentiable in a neighborhood of $\{\ph,-\ph\}$, then, after decreasing $\delta$ if necessary, the segment $t\sigma\ph+\theta v$, $0\le\theta\le1$, lies in this neighborhood.
Then the fundamental theorem of calculus gives
$$
 \Fcal(t\sigma\ph+v) = \Fcal(t\sigma\ph+v)-\Fcal(t\sigma\ph)
 =\int_0^1\dual{\Fcal'(t\sigma\ph+\theta v)}{v}\,d\theta.
$$
Using the estimate~\eqref{eq:derivative-A}, we derive \eqref{eq:upper-A}.
\end{proof}

The following result verifies Assumption~\ref{ass:approach} with $\kappa=1/2$ when $\Fcal$ has a nonzero derivative at $\sigma \ph$.

\begin{proposition}\label{prop:approach-criterion}
Let $p\ge2$ and let Assumption~\ref{ass:resonance} hold.
Suppose that $\Fcal$ is Fr\'echet differentiable at $\sigma\ph$ for some $\sigma\in\{-1,1\}$ and that there exists $z\in\W$ such that
$\dual{\Fcal'(\sigma\ph)}{z}>0$.
Then Assumption~\ref{ass:approach} is satisfied with $\kappa=1/2$.
\end{proposition}

\begin{proof}
Set $d=\dual{\Fcal'(\sigma\ph)}{z}>0$ and $v_\tau=\sigma\ph+\tau z$ for $\tau>0$.
By Assumption~\ref{ass:resonance}, $\Fcal$ vanishes on $\R\ph$, and hence its derivative at $\sigma\ph$ vanishes in every direction belonging to $\R\ph$.
This implies that $z\notin\R\ph$, and hence $v_\tau\notin\R\ph$ and $H_{\lambda_1}(v_\tau)>0$ for every $\tau>0$.

For $p\ge2$, the following estimate is well known: there exists $C_p>0$ such that
$$
 0\le |a+b|^p-|a|^p-p|a|^{p-2}a\cdot b
 \le C_p\left(|a|^{p-2}|b|^2+|b|^p\right),
 \quad a,b\in\R^N.
$$
Indeed, the lower bound follows from the convexity of $a\mapsto|a|^p$, and the upper bound follows from Taylor's formula with integral remainder and the bound $|D^2(|a|^p)|\le p(p-1)|a|^{p-2}$. 
Applying these estimates with $a=\sigma\nabla\ph$, $b=\tau\nabla z$, and $a=\sigma \ph$, $b=\tau z$, and noting that $H_{\lambda_1}(\sigma\ph)=0$ and $H_{\lambda_1}'(\sigma\ph)=0$, we obtain
\begin{align}
 0<H_{\lambda_1}(v_\tau)
 &=
 H_{\lambda_1}(\sigma\ph+\tau z)
 -
 H_{\lambda_1}(\sigma \ph)
 -
 \left<H_{\lambda_1}'(\sigma \ph), \tau z \right>\\ 
 &\le C_p\left(\tau^2\int_\Om|\nabla\ph|^{p-2}|\nabla z|^2\,dx
 +\tau^p\|\nabla z\|_p^p\right)\\
 \label{eq:lasteqh}
 &\le C_p\left(\tau^2\|\nabla z\|_p^2+\tau^p\|\nabla z\|_p^p\right)
 \le C_z\tau^2,\quad 0<\tau\le1,
\end{align}
where $C_z>0$ depends on $z$, and the third inequality follows from the H\"older inequality and $\|\nabla\ph\|_p=1$ (see, equivalently, \cite[Eq.~(2.9)]{BobkovKolonitskii2023}).
On the other hand, the assumed differentiability of $\Fcal$ at $\sigma\ph$ and the equality $\Fcal(\sigma \ph)=0$ give $\Fcal(v_\tau)=d\tau+o(\tau)$.
Thus, there exists $\tau_0\in(0,1]$ such that
\begin{equation}\label{eq:lasteqf}
 \Fcal(v_\tau)\ge\frac d2\tau
 \quad\text{and}\quad
 \frac12\le\|\nabla v_\tau\|_p\le2
 \quad \text{for any}~ 0<\tau<\tau_0.
\end{equation}
For $\eps \in (0,\tau_0^2)$, define
$w_\eps=v_{\sqrt\eps}/\|\nabla v_{\sqrt\eps}\|_p$, so that $w_\eps \in S$.
By the homogeneity of $H_{\lambda_1}$ and $\Fcal$, we deduce from \eqref{eq:lasteqh} and \eqref{eq:lasteqf} that 
$$
 0<H_{\lambda_1}(w_\eps)\le2^p C_z\eps
 \quad\text{and}\quad
 \Fcal(w_\eps)\ge\frac{d}{2^{q+1}}\eps^{1/2},
$$
which is required in Assumption~\ref{ass:approach} with $\kappa=1/2$.
\end{proof}

\begin{corollary}\label{cor:linear-source}
Let $p\ge2$ and let Assumption~\ref{ass:domain-regularity} hold.
If $f\in D_{\ph}^*\setminus\{0\}$ satisfies $f[\ph]=0$, then the functional $\Fcal(u)=f[u]$, $u\in\W$, satisfies Assumption~\ref{ass:resonance} with $q=1$, Assumption~\ref{ass:upper} with any $\kappa>0$, and Assumption~\ref{ass:approach} with $\kappa=1/2$.
\end{corollary}
\begin{proof}
For $p>2$, the H\"older inequality and \eqref{eq:phi-normalization} give
\begin{equation}\label{eq:vphi1}
 \|v\|_{\ph}^2\le\|\nabla\ph\|_p^{p-2}\|\nabla v\|_p^2=\|\nabla v\|_p^2
 \quad\text{for any }v\in\W.
\end{equation}
For $p=2$, \eqref{eq:vphi1} holds with equality by the definition of $\|v\|_{\ph}$.
Thus, $f$ restricts to a continuous linear functional on $\W$, and $f[\ph]=0$ implies that $\Fcal$ satisfies Assumption~\ref{ass:resonance} with $q=1$.
For the decomposition and parameters in Proposition~\ref{prop:upper-criterion}, we have
$$
 \Fcal(t\sigma\ph+v)=f[v]
 \le\|f\|_{D_{\ph}^*}\|v\|_{\ph}
 \le\|f\|_{D_{\ph}^*}\left(\|v\|_{\ph}^2+\|\nabla v\|_p^p\right)^{\kappa}
$$
for any $\kappa>0$. 
Hence, Proposition~\ref{prop:upper-criterion} verifies Assumption~\ref{ass:upper} with such $\kappa$.
Since $f\ne0$ and $\W$ is dense in $D_{\ph}$, there exists $z\in\W$ such that $f[z]>0$.
The derivative of $\Fcal$ at $\ph$ is $f$, so Proposition~\ref{prop:approach-criterion} verifies Assumption~\ref{ass:approach} with $\kappa=1/2$.
\end{proof}

\section{Applications}\label{sec:applications}
In this section, we discuss applications of Theorems~\ref{thm:geometry} and \ref{thm:three} to double-phase functionals and the nonlinear Fredholm alternative. 

\subsection{Double-phase functionals}\label{sec:double-phase}
For $q>1$ and $a,c\in L^\infty(\Om)$ with $c\ge0$, consider the functional
\begin{equation}\label{eq:model-F}
\Fcal(u) =
\frac1q \int_\Om \bigl( a(x)|u|^q-c(x)|\nabla u|^q\bigr)\,dx,
\end{equation}
so that the corresponding energy functional $E_{\lambda}$ has the form
$$
E_{\lambda}(u)
=
 \int_\Om\left(\frac1p|\nabla u|^p
 +\frac{c(x)}q|\nabla u|^q\right)dx
 -
 \int_\Om\left(\frac{\lambda}{p}|u|^p
  +\frac{a(x)}q|u|^q\right)dx,
$$
and its critical points are (weak) solutions of
\begin{equation}\label{eq:double1}
 -\operatorname{div}\bigl(|\nabla u|^{p-2}\nabla u
 +c(x)|\nabla u|^{q-2}\nabla u\bigr)
 =\lambda|u|^{p-2}u+a(x)|u|^{q-2}u
 \quad\text{in }\Om.
\end{equation}
The operator on the left-hand side of \eqref{eq:double1} is of double-phase type; see, for example, \cite{ColomboMingione2015,Zhikov1987}. 
The case $c\equiv0$ corresponds to the resonant indefinite functional considered in \cite{BobkovTanaka2023Indefinite}.
The case $c\equiv1$ and $a\equiv\beta_*$, where $\beta_*=\|\nabla\ph\|_q^q/\|\ph\|_q^q$, corresponds to the resonant $(p,q)$-Laplacian functional studied in \cite{BobkovTanaka2018,BobkovTanaka2022Multiplicity}.

\begin{lemma}\label{lem:model-upper}
Let $p\ge2q$, $q>1$, $a,c\in L^\infty(\Om)$ with $c\ge0$, and let Assumptions~\ref{ass:resonance} and \ref{ass:domain-regularity} hold.
Then $\Fcal \in C^1(\W,\mathbb{R})$, $\Fcal$ is sequentially weakly upper semicontinuous, and $\Fcal$ satisfies Assumption~\ref{ass:upper} with $\kappa=1/2$.
\end{lemma}
\begin{proof}
Since the map $h\mapsto|h|^{q-2}h$ is continuous from $L^q(\Om;\R^d)$ to $L^{q'}(\Om;\R^d)$, $d \geq 1$, the continuous embedding of $W_0^{1,p}(\Omega;\R^d)$ into $L^q(\Om;\R^d)$ gives the continuity of the derivative
\begin{equation}\label{eq:deriv1}
 \dual{\Fcal'(u)}{v}=\int_\Om a|u|^{q-2}uv\,dx-\int_\Om c|\nabla u|^{q-2}\nabla u\cdot\nabla v\,dx,
 \quad u,v\in\W.
\end{equation}
This means that $\Fcal \in C^1(\W,\mathbb{R})$. 
Let $u_n\weakto u$ in $\W$. Since $\Om$ is bounded, we have $\nabla u_n\weakto\nabla u$ in $L^q(\Om;\R^N)$. 
Noting that the multiplication by $c^{1/q}$ is a bounded linear map from $L^p(\Om;\R^N)$ to $L^q(\Om;\R^N)$ by the H\"older inequality, we see that $c^{1/q}\nabla u_n\weakto c^{1/q}\nabla u$ in $L^q(\Om;\R^N)$. 
In addition, the compact embedding of $\W$ into $L^q(\Omega)$ gives $u_n \to u$ in $L^q(\Omega)$.
Consequently, $\Fcal(u) \geq \limsup_{n \to +\infty} \Fcal(u_n)$, which is the desired sequential weak upper semicontinuity.

In order to justify Assumption~\ref{ass:upper}, let us verify the estimate \eqref{eq:derivative-A} from Proposition~\ref{prop:upper-criterion}. 
Take $t,\sigma,v$ as in Proposition~\ref{prop:upper-criterion}, let $0\le\theta\le1$, and set $w=t\sigma\ph+\theta v$.
Then $|t|\le2$, $|w|\le2\ph+|v|$, and $|\nabla w|\le2|\nabla\ph|+|\nabla v|$ a.e.\ in $\Omega$.
Using \eqref{eq:deriv1}, we obtain
\begin{align*}
 |\dual{\Fcal'(w)}{v}|
 &\le C\int_\Om |a|(|\ph|+|v|)^{q-1}|v|\,dx\\
 &\quad+C\int_\Om c(|\nabla\ph|+|\nabla v|)^{q-1}
 |\nabla v|\,dx=:CI_0+CI_1,
\end{align*}
where $C>0$ does not depend on $t,\sigma,v$. 
In view of Assumption~\ref{ass:domain-regularity}, we have $\ph \in C^{1,\beta}(\overline{\Omega})$ for some $\beta \in (0,1)$ by \cite{Lieberman}. 
Since $p\ge2q$, we have $2(q-1)<p$.
Thus, taking $\delta>0$ sufficiently small and using the Poincar\'e inequality and \eqref{eq:weighted-L2}, we obtain
$$
 I_0
 \le \|a\|_\infty
 \left(\int_\Om(\ph+|v|)^{2(q-1)}\,dx\right)^{1/2}\|v\|_2
 \le C\|v\|_{\ph}
 \le C\left(\|v\|_{\ph}^2+\|\nabla v\|_p^p\right)^{1/2}.
$$
To estimate $I_1$, we first show that
\begin{equation}\label{eq:double2}
 \int_\Om c^2|\nabla\ph|^{-(p-2q)}\,dx<+\infty.
\end{equation}
Since $p-2q<p-1$, \cite[Theorem~C.5]{BrascoLindgren2023} gives $\int_\Om|\nabla\ph|^{-(p-2q)}\,dx<+\infty$ under Assumption~\ref{ass:domain-regularity} in the case $N \geq 2$. 
For $N=1$, the same integrability follows from the local behavior of the $p$-sine derivative at its critical point, see \cite[p.~1234, proof of Proposition~11]{BobkovTanaka2018}.
Since $c\in L^\infty(\Om)$, these results yield \eqref{eq:double2}. 
Now, by the H\"older inequality we get
\begin{equation}\label{eq:I11}
 I_1
 \le\left(\int_\Om
 (|\nabla\ph|+|\nabla v|)^{p-2}|\nabla v|^2\,dx\right)^{1/2} 
 \left(\int_\Om
 \frac{c^2}{(|\nabla\ph|+|\nabla v|)^{p-2q}}\,dx\right)^{1/2}.
\end{equation}
The second term on the right-hand side of \eqref{eq:I11} is bounded by \eqref{eq:double2}, and the first term is bounded by $C(\|v\|_{\ph}^2+\|\nabla v\|_p^p)^{1/2}$. 
Applying Proposition~\ref{prop:upper-criterion}, we finish the proof.
\end{proof}

If we suppose, in addition to the assumptions of Lemma~\ref{lem:model-upper}, that there exists $z \in \W$ such that 
\begin{equation}\label{eq:model-derivative}
 \dual{\Fcal'(\ph)}{z}
 =\int_\Om a\ph^{q-1}z\,dx
 -\int_\Om c|\nabla\ph|^{q-2}\nabla\ph\cdot\nabla z\,dx > 0,
\end{equation}
then Proposition~\ref{prop:approach-criterion} verifies Assumption~\ref{ass:approach} with $\kappa=1/2$. 
The inequality \eqref{eq:model-derivative} says that $\ph$, being the first eigenfunction of the $p$-Laplacian, is not an eigenfunction of the weighted $q$-Laplacian. 
In this case, Theorem~\ref{thm:geometry} (with $p \geq 2q$) is applicable to $E_\lambda$.

Let us now consider Theorem~\ref{thm:three} for $E_\lambda$ with the perturbation
\begin{equation}\label{eq:model-perturbation}
 \widetilde{\Fcal}(u)=\frac1q\int_\Om b(x)|u|^q\,dx,
 \quad \text{where}~ b\in L^\infty(\Om).
\end{equation}
The formula~\eqref{eq:deriv1}, with $a$ replaced by $a+\mu b$, shows that $\Fcal_\mu\in C^1(\W,\R)$ for every $\mu\ge0$.
The sequential weak upper semicontinuity argument from the proof of Lemma~\ref{lem:model-upper} also applies with $a+\mu b$ in place of $a$, and hence the assumption~\ref{thm:three:as1} of Theorem~\ref{thm:three} is satisfied.
The assumption~\ref{thm:three:as2} follows by the standard arguments, see, e.g., \cite[Section~2]{DrabekRobinson1999}.
Let us verify the assumption on $\Acal_\mu$ of Theorem~\ref{thm:three}. 

\begin{lemma}\label{lem:model-sign-path}
Let $p > q>1$ and $a,b,c\in L^\infty(\Om)$ with $c\ge0$. 
If $\Fcal(\ph)=0$ and $\widetilde{\Fcal}(\ph)>0$, then, for any $\mu>0$ and  $u\in\Acal_\mu$, one of $u$ and $-u$ can be joined to $\ph$ by a continuous path in $\Acal_\mu$.
\end{lemma}
\begin{proof}
For $u\in\Acal_\mu$, write $u=u^+-u^-$, where $u^\pm=\max\{\pm u,0\}$.
The disjointness of supports of $u^+$ and $u^-$ gives $\Fcal_\mu(u)=\Fcal_\mu(u^+)+\Fcal_\mu(u^-)>0$, and hence at least one of $u^+$ and $u^-$ belongs to $\Acal_\mu$.
If $\Fcal_\mu(u^+)>0$, then we use the path $u^+-t u^-$, $t \in [0,1]$, to join $u^+$ to $u$. 
It is not hard to see that this path belongs to $\Acal_\mu$. 
The case $\Fcal_\mu(u^-)>0$ is handled similarly.

If $v,w\in\Acal_\mu$ are nonnegative, then the path $\xi_s=((1-s)v^q+sw^q)^{1/q}$, $s\in[0,1]$, satisfies the hidden convexity estimate \cite[Proposition~2.6]{BrascoFranzina2014}:
$$
|\nabla\xi_s|^q
\le(1-s)|\nabla v|^q+s|\nabla w|^q
 \quad\text{a.e. in } \Om,
$$
which yields
$$
 \Fcal_\mu(\xi_s)
 \ge(1-s)\Fcal_\mu(v)+s\Fcal_\mu(w)>0
 \quad \text{for any}~ s \in [0,1].
$$
Moreover, it is not hard to see that the path $s \mapsto \xi_s$ is continuous in $\W$.
Taking $w=\ph$ and concatenating with the path obtained above, we conclude that one of $u$ and $-u$ can be continuously joined to $\ph$ in $\Acal_\mu$, through its positive or negative part, for every $u\in\Acal_\mu$.
\end{proof}

Under the assumptions of Lemmas~\ref{lem:model-upper}, \ref{lem:model-sign-path}, and \eqref{eq:model-derivative}, we can apply Theorem~\ref{thm:three} to $\Fcal_\mu$.
Consequently, that theorem provides three nonzero critical points in a left neighborhood of $\lambda_1$ and two critical points in its right neighborhood, in appropriate regions of the parameter plane.

\subsection{Nonlinear Fredholm alternative}\label{sec:nonlinear-fredholm}

Recall that it was proved in \cite[Theorems~2.6 and~2.7]{Takac2006} that, for $p>2$, the equation \eqref{eq:intro-fredholm} possesses at least three solutions in a punctured neighborhood of $\lambda_1$ when the source function is a perturbation of $f$ satisfying \eqref{eq:intro-orthogonality}.
The following result weakens the regularity of the source functions and provides more explicit relation between the parameters. 
On the other hand, we do not state the existence of at least \textit{three} solutions in a right neighborhood of $\lambda_1$. 

\begin{proposition}\label{prop:fredholm}
Let $p>2$, let Assumption~\ref{ass:domain-regularity} hold, let $f_0\in D_{\ph}^*\setminus\{0\}$ be such that $f_0[\ph]=0$, 
and let $f_1\in\W^*$ satisfy $f_1[\ph]>0$.
Consider the functional
\begin{equation}\label{eq:fredholm-energy}
 J_{\lambda,\mu}(u)=\frac1pH_\lambda(u)-f_0[u]-\mu f_1[u], \quad u \in \W.
\end{equation}
Then the following assertions hold:
\begin{enumerate}[label=\textup{(\roman*)}]
\item\label{eq:fredholm-energy:1} At $(\lambda,\mu)=(\lambda_1,0)$, $J_{\lambda_1,0}$ is bounded from below, has a negative infimum, and possesses a nonzero global minimizer.
\item\label{eq:fredholm-energy:2} There exist $\mu_*,c_*>0$ such that, whenever $\mu \in (0, \mu_*)$ and $\lambda \in (\lambda_1-c_*\mu^{p}, \lambda_1)$, $J_{\lambda,\mu}$ has at least three distinct nonzero critical points. 
\item\label{eq:fredholm-energy:3} With $\mu_*$ as in \ref{eq:fredholm-energy:2}, there exists $\delta_*>0$ such that, whenever $\mu \in (0, \mu_*)$ and $\lambda \in (\lambda_1, \lambda_1+\delta_*)$, $J_{\lambda,\mu}$ has at least two distinct nonzero critical points. 
\end{enumerate}
\end{proposition}
\begin{proof}
Under the imposed assumptions, we have $q=1$, $\Fcal(u)=f_0[u]$, and $\widetilde{\Fcal}(u)=f_1[u]$.
Corollary~\ref{cor:linear-source} shows that $\Fcal$ satisfies Assumptions~\ref{ass:resonance}, \ref{ass:upper}, \ref{ass:approach} with $\kappa=1/2$.
Corollary~\ref{cor:linear-source} also shows that $f_0$ restricts to a continuous linear functional on $\W$.
Since the weak convergence in $\W$ implies the strong convergence in $L^p(\Om)$, the convexity of $u\mapsto\|\nabla u\|_p^p$ and the weak continuity of $f_0$ and $f_1$ imply that $J_{\lambda,\mu}$ is sequentially weakly lower semicontinuous.
As $p>2$, Theorem~\ref{thm:geometry}~\ref{item:geometry-resonant}~\ref{item:geometry-supercritical-order} proves \ref{eq:fredholm-energy:1}.
Adopting the arguments from \cite[Section~2]{DrabekRobinson1999}, we obtain the Palais--Smale condition for $J_{\lambda,\mu}$. 
Finally, the set $\Acal_\mu=\{u:(f_0+\mu f_1)[u]>0\}$ 
is convex and hence path-connected.
Therefore, Theorem~\ref{thm:three} applies, which gives the desired conclusion. 
\end{proof}

\end{document}